\documentclass[10pt]{amsart}

\usepackage{amsfonts,amssymb,amsmath,amscd,amstext}
\usepackage[paper size={154mm,230mm},left=6mm,right=6mm,top=10mm,bottom=10mm]{geometry}
\usepackage[colorlinks=true,linkcolor=blue,citecolor=red]{hyperref}
\usepackage[utf8]{inputenc}

\usepackage{tikz}

\usepackage{color}
\definecolor{grey}{rgb}{.7,.7,.7}

\renewcommand{\le}{\leqslant}
\renewcommand{\ge}{\geqslant}
\newcommand{\ptl}{\partial}

\newcommand{\la}{\lambda}
\newcommand{\hh}{{\mathcal{H}}}

\newcommand{\esf}{\mathbb{S}}

\newcommand{\B}{\mathcal{B}}

\newcommand{\norm}[1]{\left\lVert#1\right\rVert}

\newcommand{\vv}{\mathcal{V}}
\newcommand{\pp}{\mathcal{P}}
\newcommand{\nn}{\mathbb{N}}

\newcommand{\Sg}{\Sigma}
\newcommand{\sg}{\sigma}
\newcommand{\vsg}{\varsigma}

\newcommand{\eps}{\varepsilon}
\newcommand{\ga}{\gamma}

\newcommand{\La}{\Lambda}
\newcommand{\de}{\delta}
\newcommand{\be}{\beta}

\newcommand{\LLL}{\mathcal{L}}
\newcommand{\NN}{\mathcal{N}}

\DeclareMathOperator{\apjac}{apJac}
\DeclareMathOperator{\ric}{Ric}
\DeclareMathOperator{\sym}{sym}
\DeclareMathOperator{\spec}{Spec}
\DeclareMathOperator{\reg}{reg}

\newtheorem{theorem}{Theorem}[section]
\newtheorem{proposition}[theorem]{Proposition}
\newtheorem{lemma}[theorem]{Lemma}

\theoremstyle{definition}

\newtheorem{remark}[theorem]{Remark}

\theoremstyle{remark}

\newenvironment{enum}{\begin{enumerate}
}{\end{enumerate}}

\numberwithin{equation}{section}

\begin{document}

\title{Isoperimetric regions in anisotropically scaled product manifolds}
\author{Efstratios Vernadakis}
\address{Department of Mathematics and Statistics, University of Cyprus, 1678 Nicosia, Cyprus}
\email{vernadakis.efstratios@ucy.ac.cy}

\date{\today}
\begin{abstract}
Let $M, N$ be compact Riemannian manifolds.
Then, for fixed volume fraction, in the product of a sufficiently small homothetic copy of $M$ with $N$, every
isoperimetric region is the product of $M$ with an isoperimetric region in $N$, provided the boundaries of the isoperimetric regions in $N$ are smooth.
\end{abstract}


\maketitle


\thispagestyle{empty}

\bibliographystyle{amsplain}

\section{Introduction}

The classical isoperimetric problem seeks to determine, within a given ambient space, the subsets of prescribed volume that minimize boundary measure.
It is one of the most fundamental variational problems in geometry.

In products of Euclidean and hyperbolic spaces, W.–T. and W.–Y. Hsiang~\cite{hsiang} provided a complete description of isoperimetric hypersurfaces.
In products of a circle with a model space, isoperimetric regions were classified by Pedrosa and Ritoré \cite{peri}.
Morgan~\cite{Morpr} established lower bounds for the isoperimetric profile of a Riemannian product in terms of concave lower bounds for the profiles of the factors.
In the spherical case, Pedrosa~\cite{Pedrosa2004} classified the isoperimetric regions in the spherical cylinder $\mathbb{S}^n\times\mathbb{R}$.
In Riemannian cylinders $M\times\mathbb{R}$, Duzaar and Steffen~\cite{du-st} proved that large-volume minimizers are slabs.
This was later extended to $M\times\mathbb{R}^k$, where large isoperimetric regions are products of the first factor with geodesic balls in the second factor, see~\cite{manstr}.

A natural question was originally proposed by M. Hutchings — as reported by F. Morgan in~\cite{morpol} — who conjectured that under anisotropic scaling, the isoperimetric regions should align with a cylindrical product structure. More concretely, if one rescales one factor of a product, do minimizers eventually acquire a cylindrical structure of the form $M \times S$, where $S$ is an isoperimetric region in the other factor, once the scaling parameter is sufficiently small (or large)?

In this paper, we address this conjecture under the assumption that the isoperimetric hypersurfaces of the second factor \(N\) are of class \(C^{2,\alpha}\), a condition that holds, for example, when \(\dim N\le 8\) or for small volumes (and their complements).

We begin by recalling the scaling identities for Hausdorff measure, volume, and perimeter, and fix notation for anisotropic homotheties on product manifolds.
In the Appendix we adapt the Ros–Morgan symmetrization on horizontal slices, showing that it suffices to treat the case $\mathbb{S}^m\times N$ (Proposition~\ref{prp:reduce})
\,(see also the related symmetrization results of Morgan--Howe--Harman~\cite{mosym}).
We then analyze \emph{cylinderoids} $M_1\times S$ and symmetrized competitors in $\esf^m\times N$: compactness and slice estimates imply that, under the anisotropic deformation, isoperimetric sets subconverge (in both $L^1$ and Hausdorff topologies) to cylinderoids whose projections onto $N$ are isoperimetric (Propositions~\ref{prp:pr} and~\ref{prp:prpc}). Moreover, by standard regularity theory for perimeter minimizers,
all relevant competitors have $C^{2,\alpha}$ boundaries.  
This allows us to restrict the analysis to $C^{2,\alpha}$ graphical 
perturbations of $\Sigma_t$.
Next, exploiting the product spectral splitting and the scaling of eigenvalues, we establish a stability inheritance result: if $\Sg\subset N$ is a stable constant mean curvature hypersurface, then $tM\times\Sg$ is stable for all sufficiently small $t>0$ (Proposition~\ref{Thm:sts}).
Finally, assuming the $C^{2,\alpha}$ regularity of isoperimetric hypersurfaces in $N$, we adapt the Lyapunov–Schmidt reduction framework developed in~\cite{rqi} to the present product setting around $\Sg_t:=tM\times\Sg$. 
This analytic decomposition into a Jacobi–kernel component and its $L^2$–orthogonal complement, combined with stability estimates, yields that all nearby stationary graphs are parametrized by a finite-dimensional family of kernel deformations. This rigidity ultimately forces minimizers close to $\Sg_t$ to be cylinderoids (Theorem~\ref{thm:main}).


We begin by recalling the scaling relations for volume and perimeter, and by introducing anisotropic product scalings. We also summarize standard results on existence, regularity, and stability of isoperimetric regions in compact Riemannian manifolds.

Let $(M,g)$ be a Riemannian manifold. For a measurable set $S \subset M$, we denote by 
\[
  \vv_g(S) \quad \text{and} \quad \pp_g(S)
\]
the Riemannian volume and perimeter, respectively. The $\alpha$-dimensional Hausdorff measure with respect to $g$ is written $\hh^\alpha_g(S)$. A direct computation from the definitions yields the following scaling relations:
\begin{equation}\label{eq:est}
  \hh_{t^2 g}^{\alpha}(S) = t^{\alpha}\,\hh_g^{\alpha}(S), 
  \qquad \vv_{t^2 g}(S) = t^m\,\vv_g(S), 
  \qquad \pp_{t^2 g}(S) = t^{m-1}\,\pp_g(S),
\end{equation}
where $m = \dim M$.  
Throughout the manuscript, whenever no ambiguity arises, we shall drop subscripts and superscripts in the notation. We write $tM$ for the Riemannian manifold $(M,t^2 g)$, and refer to it as the \emph{$t$-homothety} of $(M,g)$.

Let $(M_i,g_i)$, $i=1,2$, be Riemannian manifolds, and set $g=g_1 \times g_2$.  
We define two anisotropic product scalings:
\begin{itemize}
  \item The \emph{right anisotropic $t$-homothety} is
  \[
    (M_1 \times M_2, g_1 \times t^2 g_2),
  \]
  denoted by $(M_1 \times M_2, g_+^t)$ or simply $M_1 \times tM_2$.
  \item The \emph{left anisotropic $t$-homothety} is
  \[
    (M_1 \times M_2, t^2 g_1 \times g_2),
  \]
  denoted by $(M_1 \times M_2, g_-^t)$ or simply $tM_1 \times M_2$.
\end{itemize}
These anisotropic deformations play a central role in our analysis, as they geometrically encode the effect of collapsing or expanding one factor of a product manifold. We record the following scaling behavior for volume and perimeter under right anisotropic homotheties.

\begin{lemma}\label{lem:anest}
Let $(M_i,g_i)$ be Riemannian manifolds of dimensions $m_i$ ($i=1,2$), and set $g=g_1\times g_2$. Then the following hold for all measurable $S\subset M_1\times M_2$ and $(m_1+m_2-1)$–rectifiable $\Sigma$:
\begin{enumerate}
\item[(i)] For every set $S \subset M_1 \times M_2$,
\begin{equation}\label{eq:anest1}
  \vv_{g_+^t}(S) = t^{m_2} \,\vv_g(S).
\end{equation}
\item[(ii)] If $\Sg \subset M_1 \times M_2$ is $(m_1+m_2-1)$-rectifiable, then for $t \leq 1$,
\begin{equation}\label{eq:anest2}
  \hh_{g_+^t}^{m_1+m_2-1}(\Sg) 
  \leq t^{m_2-1}\,\hh_g^{m_1+m_2-1}(\Sg).
\end{equation}
\item[(iii)] Equality holds in \eqref{eq:anest2} if and only if the $g$-normal of $\Sg$ is tangent to $M_2$, up to a $\hh_g^{m_1+m_2-1}$-null set.
\end{enumerate}
\end{lemma}

\begin{proof}
Part (i) follows directly from computing the Jacobian of the identity map from $(M_1 \times M_2,g)$ to $(M_1 \times M_2,g_+^t)$.  
For (ii), at a regular point of $\Sg$ the $g$-unit exterior normal vector $\nu$ can be decomposed as $\nu = a\nu_1 + b\nu_2$ so that $a^2+b^2=1$ and $\nu_1$, $\nu_2$ unit vectors tangent to $M_1$ and $M_2$, respectively. The Jacobian of the identity map from $(M_1 \times M_2,g)$ to $(M_1 \times M_2,g_+^t)$ restricted to $\Sg$ then equals $t^{m_2-1}(t^2a^2+b^2)^{1/2}$. This yields inequality \eqref{eq:anest2}.  
Assertion (iii) follows from the fact that $\nu$ is tangent to $M_2$ precisely when $a=0$ almost everywhere.  

\end{proof}
The following local isoperimetric inequalities are classical, see Duzaar and Steffen \cite{du-st}.
\begin{lemma}
\label{lem:ine}
Let $M$ be a compact $m$-dimensional Riemannian manifold. Given $0<v_0<\hh^m(M)$, there exist positive constants  $\ga_1(M,v_0),\ga_2(M,v_0)$ such that for any set $S\subset M$ with $(m-1)$-rectifiable boundary and $0<\hh^m(S)<v_0$ the following isoperimetric inequalities hold
\begin{equation}
\label{eq:Ine1}
\hh^{m-1}(\ptl S)\ge \ga_1(M,v_0)\,\hh^m(S)
\end{equation}
and
\begin{equation}
\label{eq:Ine2}
\hh^{m-1}(\ptl S)\ge \ga_2 (M,v_0)\,\hh^m(S)^{(m-1)/m},
\end{equation}
\end{lemma}

The \emph{isoperimetric profile} of a Riemannian manifold $(M,g)$ is defined by
\[
  I_g(v) = \inf \big\{ \pp_g(E) : \vv_g(E) = v \big\}.
\]
A set $E \subset M$ is called an \emph{isoperimetric region} if $\pp_g(E) = I_g(\vv_g(E))$. Existence of isoperimetric regions in compact manifolds is guaranteed by standard compactness arguments in geometric measure theory, see \cite{Maggi}. Moreover, classical regularity theory implies that the boundary of an isoperimetric
region is a smooth hypersurface except for a singular set of Hausdorff codimension at least $8$,  as proved by Morgan \cite{moreg} and by Gonzales–Massari–Tamanini \cite{gomato}. In Chapter 13.2 of \cite{Morgan2016}, Morgan, following methods of Smale \cite{Smale1999},  describes constructions of irregular isoperimetric hypersurfaces in manifolds of dimension greater than seven. Furthermore, Morgan–Johnson \cite{morjohn} and Nardulli \cite{nard} show that, for small volumes (and their complements), isoperimetric hypersurfaces are smooth.

Isoperimetric regions are considered up to sets of measure zero. 
In particular, every isoperimetric set is equivalent to its closure, and its perimeter can be computed on the topological boundary:
\begin{equation}\label{eq:ww4}
\pp(E) = \pp(\overline{E}) = \hh^{m-1}(\ptl E).
\end{equation}

Let $\Sg\subset M$ be an isoperimetric hypersurface. Standard first and second variation computations
(see for instance Barbosa--do Carmo--Eschenburg~\cite{bd})
imply that the regular part $\reg(\Sg)$ has constant mean curvature and satisfies a stability inequality.
Specifically, for every smooth compactly supported function $u:\reg(\Sg)\to\mathbb{R}$ with zero mean one has
\begin{equation}\label{eq:stab}
- \int_\Sg u \,\big(\Delta u + \ric(\nu,\nu)u + |\sg|^2 u\big) \,\geq 0,
\end{equation}

where $\nu$ is the exterior unit normal to $\Sg$, $\Delta$ is the Laplace operator on $\Sg$, $|\sg|^2$ is the squared norm of the second fundamental form, and $\ric$ denotes the Ricci tensor of the ambient manifold $M$. 
We denote by $J_\Sg$ the Jacobi operator, which acts on smooth functions $u$ as follows:
\[
  J_\Sg u = \Delta u + \ric(\nu,\nu)\,u + |\sg|^2 u.
\]
If the inequality in \eqref{eq:stab} is strict for all nontrivial $u$, we say that $\sg$ is \emph{strictly stable}.

\section{Cylinderoids} 

In this section, we analyze \emph{cylinderoids}, that is, sets of the form $M_1\times S$, or equivalently $E=M_1\times \pi_2(E)$ up to sets of measure zero, and their role in the isoperimetric problem.
Using the reduction proved in the Appendix, any instance relevant to our main theorem can be placed, up to volume-perimeter preserving identifications, inside a product $\esf^m\times N$.
This reduction justifies working with spherically \emph{symmetrized} sets, since replacing each horizontal slice by a concentric geodesic ball in $\esf^m$ preserves volume and decreases perimeter.

We begin by collecting several elementary lemmas on convergence and measure-theoretic properties of cylinderoids and symmetrized subsets of $\esf^m\times N$. 
These results will be repeatedly used in the proof of Proposition~\ref{prp:pr}.

\begin{lemma}
\label{lem:hl}
\mbox{}
Let $M$ be a compact Riemannian manifold and $\{E_i\}_{i\in\nn}$ be a sequence of compact sets converging to $E_H$ in the Hausdorff topology and to $E_L$ in the $L^1$ topology. Then $E_L\subset E_H\,\, a.e.$
\end{lemma}
\begin{proof}
Fix $r>0$. Since $E_i\to E_H$ in the Hausdorff topology, 
\begin{equation}
\label{eq:hl0}
E_i\subset [E_H]_{r}\quad \text{for sufficiently large}\,\, i\in\nn.
\end{equation}
Where $[E_H]_{r}$ denotes the closed $r$–neighbourhood of $E_H$. Since $\chi_{E_i}\to \chi_{E_L}$ in the $L^1$ topology, then $\chi_{E_i}\to \chi_{E_L}$ a.e, for a non-relabeling subsequence. In other words  there exists a null set $O\subset M$ so that $\chi_{E_i}\to \chi_{E_L}$ pointwise in $M\setminus O$. Let $p\in E_L\setminus O$ then, as $\chi_{E_i}(p)\to \chi_{E_L}(p)$, we get $p\in E_i$ for sufficiently large $i\in\nn$. Which, by \eqref{eq:hl0}, yields $p\in [E_H]_{r}$. As $r>0$ was arbitrary we have
\begin{equation}
\label{eq:hl1}
E_L\setminus O\subset [E_H]_{r}\quad \text{for every}\quad r>0.
\end{equation}
Since $E_H$ is compact, as a Hausdorff limit of such sets , we obtain
\begin{equation}
\label{eq:hl2}
E_L\setminus O\subset \bigcap_{r>0}[E_H]_{r}=E_H.
\end{equation}
This concludes the proof.
\end{proof}

Let $M_i$ be compact  Riemannian manifolds, $E\subset M_1\times M_2$, $(x,y)\in M_1\times M_2$ we define 
\begin{equation}
\label{eq:defy}
(E)^{x}=\pi_1^{-1}(x)\cap E \quad\text{and}\quad (E)_{y}=\pi_2^{-1}(y)\cap E
\end{equation}
and
\begin{equation}
\label{eq:defew}
E{\{w\}}=\big\{(x,y)\in E:x\in M_1, y\in M_2\,\,\text{and}\,\,0<\vv_{g_1}((E)_{y})\le w\big\}.
\end{equation}

We also define the slice–volume map $\Phi_E:M_2\to [0,\vv(M_1)]$ by
\begin{equation}\label{eq:phie}
\Phi_E(y)=
\begin{cases}
\vv_{g_1}\big((E)_y\big), & \text{if } (E)_y\neq \emptyset,\\
0, & \text{if } (E)_y=\emptyset.
\end{cases}
\end{equation}
Consequently,
\begin{equation}\label{eq:cophi}
\vv(E)=\int_{M_2}\vv_{g_1}\big((E)_y\big)\,d\vv_{g_2}
=\int_{M_2}\Phi_E(y)\,d\vv_{g_2}.
\end{equation}

In what follows, $\esf$ denotes a round sphere.

\begin{lemma}
\label{lem:h=l}
\mbox{}
Let $N$ be a compact Riemannian manifold and $\{E_i\}_{i\in\nn}\subset \esf\times N$ be a sequence of compact symmetrized sets converging to $E_H$ in the Hausdorff topology and to $E_L$ in the $L^1$ topology. Then $E_L= E_H,\,\, a.e.$
\end{lemma}

\begin{proof}
Since $E_i\to E_L$ in the $L^1$ topology, we obtain.
\begin{equation}
\label{eq:h=l1}
\vv(E_i)\to \vv(E_L)
\end{equation}
As $E_L\subset E_H\,\, a.e$ by Lemma \ref{lem:hl}, it suffices to show that $\vv(E_H)= \vv(E_L)$. Since $E_i$ are compact and $E_i\to E_H$ in the Hausdorff topology, then $E_H$ is compact. Since $E_i\to E_H$ in the Hausdorff topology in $\esf\times N$ and 
$\esf\times\{y\}$ is closed for each fixed $y\in N$ we obtain
\begin{equation}
\label{eq:h=l2}
(E_i)_y\to (E_H)_y,\,\text{in the Hausdorff  topology for every}\, y\in N.
\end{equation}
By assumption $E_i$ are symmetrized hence $(E_i)_y$ are concentric geodesic balls in $\esf\times \{y\}$. As a consequence
\begin{equation}
\label{eq:h=l3}
\vv_{g_1}((E_i)_y)\to \vv_{g_1}((E_H)_y)\quad \text{for every}\quad y\in N,
\end{equation}
By definition of $\Phi$ we get that $\Phi_{E_i}\to \Phi_{E_H}$ pointwise. Observe that $\{\Phi_{E_i}\}_{i\in\nn}$ is bounded by $\vv(\esf)$. Consequently, by the dominated convergence theorem, we have that
\begin{equation}
\label{eq:h=l4}
\int_N\Phi_{E_i}d\vv_{g_2}\to \int_N\Phi_{E_H}d\vv_{g_2}.
\end{equation}
Finally, thanks to \eqref{eq:h=l1}, \eqref{eq:h=l4} and \eqref{eq:cophi}
 we obtain
\begin{equation}
\label{eq:h=l5}
\vv(E_L)=\lim\vv(E_i)=\lim\int_N\Phi_{E_i}d\vv_{g_2}= \int_N\Phi_{E_H}d\vv_{g_2}=\vv(E_H).
\end{equation}
This concludes the proof.
\end{proof}

\begin{lemma}
\label{lem:ww}
\mbox{}
Let $N$ be a compact Riemannian manifold and $\{E_i\}_{i\in\nn}\subset \esf\times N$ be a sequence of compact symmetrized sets converging to $E$ in the Hausdorff topology. If $E$ is not a cylinderoid  then there exist $\de>0$ and a $0<w<\vv(\esf)$ so that $\vv(E_i{\{w\}})>\de$ for sufficiently large $i\in\nn$. 
\end{lemma}
\begin{proof}
Since $E$ is not of the form $ \esf\times \pi_2(E)$ there exist $A_0\subset N$ with $\vv_N(A_0)>0$ and $0<w_0\le w_1<\vv(\esf)$ such that 
\begin{equation}
\label{eq:ww1}
w_0 \le \vv_{ g_1}((E)_y)\le w_1 \quad \text{for every}\quad y\in A_0.
\end{equation}
Where by the definition of $\Phi$ we get
\begin{equation}
\label{eq:ww1a}
w_0 \le \Phi_E(y)\le w_1 \quad \text{for every}\quad y\in A_0.
\end{equation}
Arguing as in the proof of Lemma \ref{lem:h=l} we obtain that the Hausdorff convergence $E_i\to E$ implies the pointwise convergence $\Phi_{E_i}\to \Phi_{E}$. Since $\{\Phi_{E_i}\}_{i\in\nn}$ is bounded, then Egorov's Theorem yields the almost uniform convergence 
$\Phi_{E_i}\to \Phi_{E}$. Consequently there exist $\eps>0$ and $A= A(\eps)\subset A_0$ so that $w_0-\eps>0$, $w_1+\eps<\vv(\esf)$ and $\vv_N(A)>0$ such that
\begin{equation} 
\label{eq:ww2}
0<w_0-\eps \le \Phi_{E_i}(y) \le w_1+\eps<\vv(\esf), \, \text{for every}\, y\in A\, \text{and}\,i\, \text{sufficiently large}.
\end{equation}
Set $w= w_1+\eps \in (0,\vv(\esf))$.  Then, by Fubini's Theorem and that $w_0-\eps$, $\vv_N(A_0)$ are positive, we finally get
\begin{equation}
\label{eq:ww4}
\vv(E_i{\{w\}})\ge\int_A\vv_{g_1}\big((E_i)_y\big)\,d\vv_{g_2}=\int_A\Phi_{E_i}(y)\,d\vv_{g_2}\ge (w_0-\eps)\vv_{g_2}(A):=\de>0,
\end{equation}
for sufficiently large $i\in\nn$. 
This concludes the proof.
\end{proof}
\begin{lemma}
\label{lem:im}
Let $(M_i,g_i)$ be $m_i$-dimensional compact Riemannian manifolds, $i=1,2$. Let $t>0$, $\be\in(0,1)$ and $v= v(t,\be)=\be\vv_{g_+^t}(M_1\times M_2)$. Then
\begin{equation}
\label{eq:im}
I_{g_+^t}(v)\le c\,v^{(m_2-1)/m_2}
\end{equation}
where $c>0$ is independent of $t$.
\end{lemma}
\begin{proof}
Let $S\subset M_2$ be an isoperimetric set of volume $\be \vv(M_2)$. Then by definition
\begin{equation}
\label{eq:Im1}
I_{g_+^t}(v)\le \pp_{g_+^t}(M_1\times S)= \vv(M_1)\,I_{t^2g_2}\big(\be \vv(M_2)t^{m_2}\big).
\end{equation}
Where we used that $v=\be\vv(M_1) \vv(M_2)t^{m_2}$, by the hypothesis. Using again this equality
and that 
\begin{equation}
\label{eq:Im2}
I_{{t^2g_2}}\big(\be \vv(M_2)t^{m_2}\big)=t^{m_2-1}I_{g_2}\big(\be \vv(M_2)\big)
\end{equation}
owing to \eqref{eq:est}, the proof follows. Where
\begin{equation}
\label{eq:Im3}
c=\vv(M_1)\,\big(\be\vv(M_1\times M_2) \big)^{-\frac{(m_2-1)}{m_2}}\,I_{g_2}\big(\be \vv(M_2)\big).
\end{equation}
Hence the constant $c$ is independent of $t$, as claimed.
\end{proof}

We are now ready to prove the main results of the section.
\begin{proposition}
\label{prp:pr}
\mbox{}
Assume $0<\be<1$, $t_i\uparrow\infty\,$ and  $\{E_i\}_{i\in\nn}$ be a sequence of compact symmetrized isoperimetric sets in $ \esf\times t_i N$ of volume $\vv_{g_+^{t_i}}(E_i)=\be \vv_{g_+^{t_i}}(\esf\times N)$. Let $t_0>0$ and $s_i=t_0\,t_i^{-1}$. 
\begin{enum}
\item
Then, possibly passing to a non-relabeling subsequence, $\{E_i\}_{i\in\nn}$ converges both in $L^1$ and Hausdorff topology of $\esf\times t_0N$,  to a finite perimeter set $E$ which is a cylinderoid, i.e $E=\esf\times \pi_2(E)$. 
\item Moreover $\pi_2(E)$ is an isoperimetric set in $t_0N$, of volume fraction $\be$.
\item Furthermore
\begin{equation}
\label{eq:ane00}
\vv_{g_+^{t_0}}(E_i)=\vv_{g_+^{t_0}}(\esf\times \pi_2(E)) \,\,\text{and}\,\,\pp_{g_+^{t_0}}(E_i)\le\pp_{g_+^{t_0}}(\esf\times \pi_2(E))
\end{equation}
\end{enum}
\end{proposition}
\begin{proof}
Owing to \eqref{eq:anest1} we get
\begin{equation}
\label{eq:thm0}
\vv_{g_+^{t_0}}(E_i)=\be\vv_{g_+^{t_0}}(\esf\times  N) :=v_0\quad\text{for every}\,\,i\in\nn. 
\end{equation}
Since $N$ is compact, there exists an isoperimetric region $S\subset N$ of volume $v_0\,\vv(\esf)^{-1}$.  As $E_i$ are isoperimetric by assumption, we have
\begin{equation}
\label{eq:cona}
\vv_{g_+^{t_i}}(\esf\times S)= \vv_{g_+^{t_i}}(E_i) \,\,\text{and}\,\,\pp_{g_+^{t_i}}(E_i)\le\pp_{g_+^{t_i}}(\esf\times S)
\end{equation}
Thanks to (ii) and (iii) of Lemma \ref{lem:anest}, we deduce 
\begin{equation}
\label{eq:conb}
\frac{\pp_{g_+^{t_0}}(E_i)}{\pp_{g_+^{t_0}}(\esf\times S)}=\frac{\pp_{g_+^{s_it_i}}(E_i)}{\pp_{g_+^{s_it_i}}(\esf\times S)}\le\frac{s_i^{n-1}\pp_{g_+^{t_i}}(E_i)}{s_i^{n-1}\pp_{g_+^{t_i}}(\esf\times S)}\le 1
\end{equation} 
Consequently, $\{\pp_{g_+^{t_0}}(E_i)\}_{i\in\nn}$ is bounded and so, possibly passing to a subsequence, $E_i\to E$  in the $L^1(\esf\times t_0 N)$ topology. Hence, by the semi-continuity of the perimeter, we get
\begin{equation}
\label{eq:con1}
\pp_{g_+^{t_0}}(E)\le\liminf\pp_{g_+^{t_0}}(E_i)\quad\text{and}\,\,\vv_{g_+^{t_0}}(E)=\be\vv_{g_+^{t_0}}(\esf\times  N) :=v_0
\end{equation}
Owing to Lemma \ref{lem:h=l}, possibly passing to a non-relabeling subsequence, we get $E_i\to E$ in the Hausdorff topology of $\esf\times t_0 N$. We argue by contradiction, assume that $E\not=\esf\times \pi_2(E)$. Then by Lemma \ref{lem:ww} there are $\de>0$ and  $w_0\in (0,\vv(\esf))$ so that 
\begin{equation}
\label{eq:con2}
\vv_{g_+^{t_0}}(E_i{\{w_0\}})>\de\quad\text{for sufficiently large}\,\,i\in\nn.
\end{equation}
Hence, by \eqref{eq:anest1}, we get
\begin{equation}
\label{eq:con3}
\frac{\vv_{g_+^{t_i}}(E_i{\{w_0\}})}{\vv_{g_+^{t_i}}(E_i)}=\frac{\vv_{g_+^{t_0}}(E_i{\{w_0\}})}{\vv_{{g_+^{t_0}}}(E_i)}= \frac{\vv_{{g_+^{t_0}}}(E_i{\{w_0\}})}{\be\vv_{{g_+^{t_0}}}(\esf\times  N)}\ge \frac{\de}{\be\vv_{{g_+^{t_0}}}(\esf\times  N)} :=c_1>0,
\end{equation}
for sufficiently large $i\in\nn$. Observe that since $ E_i\{w_0\}$ is symmetrized there holds
\begin{equation}
\label{eq:3a}
\pi_2(\ptl E_i\cap E_i\{w_0\})=\pi_2( E_i\{w_0\}) := \Xi_i.
\end{equation}
Now owing to \eqref{eq:con3} and Fubini's Theorem we obtain
\begin{equation}
\label{eq:con4}
\vv_{g_+^{t_i}}(E_i)\le c_1^{-1}\int_{\Xi_i}\hh_{g_1}^m( E_i\cap \pi_2^{-1}(y))\,d\hh_{{t_i}^2g_2}^n
\end{equation}

By the definition of $\Xi_i$ we have
\begin{equation}
\label{eq:con4a}
0<\vv_{g_1}( E_i\cap \pi_2^{-1}(y))\le w_0<\vv(\esf)\quad\text{for every}\quad y\in \Xi_i.
\end{equation}
Consequently
\begin{equation}
\label{eq:con5}
\ga_1\hh_{g_1}^m( E_i\cap \pi_2^{-1}(y))\le \hh_{g_1}^{m-1}\big(\ptl( E_i\cap \pi_2^{-1}(y))\big) \quad\text{for every}\quad y\in \Xi_i,
\end{equation}
where $\ga_1=\ga_1 (\esf,w_0)$ is an isoperimetric constant as in \eqref{eq:Ine1}. Note that
\begin{equation}
\label{eq:5a}
\pi_2\big|^{-1}_{\ptl E_i}(y)=\ptl E_i\cap{\pi_2}^{-1}(y)\supset \ptl( E_i\cap \pi_2^{-1}(y))
\end{equation}
and
\begin{equation}
\label{eq:5b}
\apjac_n(\pi_2\big|_{\ptl{E_i}})\le\apjac_n(\pi_2)= 1.
\end{equation}
By regularity, we have
\begin{equation}
\label{eq:5c}
\pp_{g_+^{t_i}}(E_i)=\hh_{g_+^{t_i}}^{m+n-1}(\ptl E_i).
\end{equation}
Applying the co-area formula, for rectifiable sets \cite[Theorem 3.2.22.]{fed} we get
\begin{equation}
\label{eq:con6}
\pp_{g_+^{t_i}}(E_i)\ge \int_{\Xi_i}\hh_{g_1}^{m-1}\big(\pi_2\big|^{-1}_{\ptl E_i}(y)\big)\,\hh_{{t_i}^2g_2}^n\ge\int_{\Xi_i}\hh_{g_1}^{m-1}\big(\ptl( E_i\cap \pi_2^{-1}(y))\big)\,\hh_{{t_i}^2g_2}^n
\end{equation}
Combining this with \eqref{eq:con4} we finally deduce
\begin{equation}
\label{eq:con7}
I_{g_+^{t_i}}\big(\vv_{g_+^{t_i}}(E_i)\big)=\pp_{g_+^{t_i}}(E_i)\ge c_1\,\ga_1\vv_{g_+^{t_i}}(E_i)\quad\text{for every}\,\,i\ge i_0.
\end{equation}
Which gives a contradiction by \eqref{eq:im}, since $\vv_{g_+^{t_i}}(E_i)\to \infty$ by the hypothesis. As a consequence, $E=\esf\times \pi_2(E)$.

In order to prove (ii) we have to show that $\pi_2(E)$ is isoperimetric. Owing to \eqref{eq:con1} we get  
\begin{equation}
\label{eq:con8}
\vv(\esf)I_{t_0^2g_2}(v_0\,\vv(\esf)^{-1}) \le\vv(\esf)\,\pp_{t_0^2g_2}(\pi_2(E))=\pp_{g_+^{t_0}}(E)\le\liminf\pp_{g_+^{t_0}}(E_i).
\end{equation}
To finish the proof we are going to obtain the inverse inequality. Passing to the limit in \eqref{eq:conb}, taking into account that $S$ is isoperimetric, we get
\begin{equation}
\label{eq:con11}
\limsup\pp_{g_+^{t_0}}(E_i)\le\pp_{g_+^{t_0}}(\esf\times S)=\vv(\esf)\,I_{t_0^2g_2}(v_0\,\vv(\esf)^{-1})
\end{equation} 
Combining \eqref{eq:con8} and \eqref{eq:con11}, we get that $\pp_{t_0^2g_2}(\pi_2(E))=I_{t_0^2g_2}(v_0\,\vv_{\esf}(\esf)^{-1})$.
Finally (iii) follows by \eqref{eq:thm0} and \eqref{eq:conb} for $S=\pi_2(E)$.
\end{proof}

\begin{remark}
\label{rem:complements}
Note that fibrewise symmetrization is compatible with taking complements:
the complement of a symmetrized slice inside a fibre is again a
symmetrized slice (now centred at the antipodal pole).  
Consequently, all arguments in the preceding lemmas and in
Proposition~\ref{prp:pr} apply equally to the complements $E_i^{\,c}$.
We use this observation in the proof of Proposition~\ref{prp:prpc}.
\end{remark}

\begin{proposition}
\label{prp:prpc}
\mbox{}
Under the hypotheses of Proposition~\ref{prp:pr} there holds
\begin{equation}
\label{eq:prpc}
\ptl E_i\to \ptl E\quad\text{in the Hausdorff topology}.
\end{equation}
\end{proposition}

\begin{proof}
We shall prove that for every $\rho>0$ there exists 
$i_0=i_0(\rho)\in\nn$ such that for every $i\ge i_0$,
\begin{equation}
\label{eq:thmc0}
\ptl E_i\subset [\,\ptl E\,]_{\rho},
\qquad\text{and}\qquad
\ptl E\subset [\,\ptl E_i\,]_{\rho}.
\end{equation}
From Proposition~\ref{prp:pr} we have
\begin{equation}
\label{eq:thmc1}
\overline{E_i}\to \overline{E}
\quad\text{in the Hausdorff topology}.
\end{equation}
Since set-complementation is continuous under $L^1$ convergence and complements of isoperimetric sets are also isoperimetric, the same proof applies to $E_i^c$, and hence
\begin{equation}
\label{eq:thmc2}
\overline{E_i^{\,c}}\to \overline{E^{\,c}}
\quad\text{in the Hausdorff topology}.
\end{equation}

\emph{First inclusion.}
Assume that the first inclusion in the statement fails.  
Then there exist $\eps>0$ and infinitely many
$x_i\in \ptl E_i$ such that
$x_i\notin [\,\ptl E\,]_\eps$.  
Passing to a subsequence, $x_i\to x$ with $x\notin 
[\,\ptl E\,]_{\eps/2}$.  
But $x_i\in \overline{E_i}\cap \overline{E_i^{\,c}}$, and by
\eqref{eq:thmc1}, \eqref{eq:thmc2} we obtain
\begin{equation}
\label{eq:thmc3}
x\in \overline{E}\cap\overline{E^{\,c}}=\ptl E,
\end{equation}
a contradiction.  
Thus the first inclusion follows.

\emph{Second inclusion.} 

Fix $\rho>0$ and argue by contradiction.  
Assume there are infinitely many $x_i\in\ptl E$ with 
$x_i\notin [\,\ptl E_i\,]_{\rho}$.  Since $x_i\notin [\,\ptl E_i\,]_{\rho}$, we have
$d(x_i,\ptl E_i)> \rho$. 

Passing to a subsequence $x_i\to x\in\ptl E$. By regularity, $x$ is a measure–theoretic boundary point of $E$. 
Choose
$r\in(0,\rho/2)$ small enough so that
\begin{equation}
\label{eq:rev.new}
\vv(B_r(x)\setminus E)>0
\quad\text{and}\quad
\vv(E\cap B_r(x))>0.
\end{equation}
Then, for all large $i\in\nn$,
\begin{equation}
\label{eq:rev1}
B_{r}(x)\cap \ptl E_i=\emptyset.
\end{equation}
Passing to a subsequence, assume that for all large $i\in\nn$ either 
$B_{r}(x)\subset E_i$ or $B_{r}(x)\subset E_i^{\,c}$.
Without loss of generality suppose that 
$B_{r}(x)\subset E_i$ for all large $i$ (the other case is treated analogously). Thus the set 
$E\setminus E_i$ does not intersect the ball.
Hence
\begin{equation}
\label{eq:rev2}
  (E_i\triangle E)\cap B_{r}(x)
  = (E_i\setminus E)\cap B_{r}(x)
  = B_{r}(x)\setminus E.
\end{equation}
Consequently \eqref{eq:rev.new} implies
\begin{equation}
\label{eq:rev3}
\vv\bigl((E_i\triangle E)\cap B_{r}(x)\bigr) 
=\vv(B_{r}(x)\setminus E)>0
\quad\text{for all large }i\in\nn,
\end{equation}
which contradicts $\chi_{E_i}\to\chi_E$ in $L^1$.
Thus the second inclusion follows.
Combining this with the first inclusion yields the claimed Hausdorff
convergence of the boundaries.

\end{proof}

\section{Eigenvalues and Stability}

In this section we recall the spectral behaviour of the Laplacian on
products under anisotropic homotheties. This description will be used
to analyse the Jacobi operator of $tM\times\Sigma$ and to show that
$tM\times\Sigma$ is stable for all sufficiently small $t>0$.

Let $(M_i,g_i)$, $i=1,2$, be compact Riemannian manifolds and $t>0$. Denote by
\begin{equation}
\label{eq:sp}
\spec(g_i)=\{\la_j(g_i)\}_{j\in\{0\}\cup\nn}
\end{equation}
the increasing sequence of eigenvalues of the Laplacian $\Delta^{g_i}$, $i=1,2$. 
If
\begin{equation}
\label{eq:eigprt}
u(x_1,x_2)=u_1(x_1)\,u_2(x_2), \qquad x_i\in M_i,
\end{equation}
and $\Delta=\Delta^{g_1\times g_2}$, then
\begin{equation}
\label{eq:delprt}
\Delta u = u_2 \, \Delta^{g_1}u_1 + u_1 \, \Delta^{g_2}u_2.
\end{equation}
It follows that
\begin{equation}
\label{eq:spprt}
\spec(g_1\times g_2) = \big\{\la_{j_1}(g_1) + \la_{j_2}(g_2)\big\}_{j_1,j_2\in\{0\}\cup\nn}
\end{equation}
and that the eigenspaces of $\spec(g_1\times g_2)$ are spanned by products of functions from the eigenspaces of $\spec(g_i)$, $i=1,2$, see Chapter II in \cite{Chaveleigen}. 

Moreover,
\begin{equation}
\label{eq:tla}
\Delta^{t^2g_i} = t^{-2}\Delta^{g_i} \qquad \text{and hence} \qquad \la_j(t^2g_i)=t^{-2}\la_j(g_i),\quad j\in\{0\}\cup\mathbb{N}.
\end{equation}

\begin{proposition}
\label{Thm:sts}
Let $\Sigma\subset N$ be a smooth constant–mean–curvature hypersurface.
If $\Sigma$ is stable, then $tM\times\Sigma$ is stable for all sufficiently small $t>0$.
\end{proposition}

\begin{proof}
Let $\bar{g}=g_1\times\bar{g}_2$, where $\bar{g}_2=g_2|_{T\Sg\times T\Sg}$. 
Note that the outer unit normal $\nu$ on $tM\times \Sg$ is tangent to the second factor. Consequently,
\begin{equation}
\label{eq:curvs}
\ric^{g_-^{t}}(\nu,\nu)=\ric^{\,g_2}(\nu,\nu), 
\qquad 
\sg^{\bar{g}_-^{t}}=\sg^{\,\bar{g}_2},
\end{equation}
where $\ric$ denotes Ricci tensor and $\sg$ second fundamental form. Hence
\begin{equation}
\label{eq:sts00}
q(x,\vsg):=\ric^{g_-^{t}}(\nu,\nu)+|\sg^{\bar{g}_-^{t}}|^2
= \ric^{\,g_2}(\nu,\nu)+|\sg^{\,\bar{g}_2}|^2 :=q(\vsg).
\end{equation}
Set
\begin{equation}
\label{eq:sts0}
t_0^2=\frac{\la_1(g_1)}{\|q\|_\infty}, \qquad t<t_0.
\end{equation}
As a consequence of the product spectral decomposition, it suffices to show that
\begin{equation}
\label{eq:sts1}
-\int_{tM\times\Sg} f\big(\Delta f+qf\big)\geq0 \qquad \text{for all smooth } f \text{ with } \int_{tM\times\Sg} f=0,
\end{equation}
where $f(x,\vsg)=u(x)v(\vsg)$, $x\in M$, $\vsg\in \Sg$ and $\Delta=\Delta^{\bar{g}_-^{t}}$. Owing to \eqref{eq:delprt}
\begin{equation}
\label{eq:sts2}
\begin{split}
-\int_{tM\times\Sg}f(\Delta f+qf)
&=-\int_{tM\times\Sg} v^2 u\Delta^{t^2g_1} u 
    -\int_{tM\times\Sg} u^2 v \Delta^{\bar{g}_2} v
    -\int_{tM\times\Sg} q u^2 v^2.
\end{split}
\end{equation}
Applying Stokes’ theorem:
\begin{equation}
\label{eq:sts3}
-\int_{tM\times\Sg} v^2 u \Delta^{t^2g_1}u 
= \int_{tM} |\nabla^{t^2g_1}u|^2_{t^2g_1}\,\int_{\Sg} v^2,
\end{equation}
and
\begin{equation}
\label{eq:sts4}
-\int_{tM\times\Sg} u^2 v \Delta^{\bar{g}_2}v 
= \int_{tM} u^2 \,\int_{\Sg} |\nabla^{\bar{g}_2}v|^2_{\bar{g}_2}.
\end{equation}

Now, two cases arise depending on the zero mean condition in \eqref{eq:sts1}

\emph{Case 1:}
 $\int_{M}u=0$. Then, by the variational characterization of the first eigenvalue and the previous identities,
\begin{equation}
\label{eq:sts2a}
\begin{split}
&-\int_{tM\times\Sg}f\big(\Delta f+qf\big)
\\
=&-\int_{tM\times\Sg}v^2u\Delta^{t^2g_1} u-\int_{tM\times\Sg}u^2v\Delta^{{\bar{g}_2}} v -\int_{tM\times\Sg}qu^2v^2
\\
\ge&-\int_{tM\times\Sg}v^2u\Delta^{t^2g_1} u -\int_{tM\times\Sg}qu^2v^2
\\
= &\int_{\Sg}v^2 \Big(\int_{tM}|\nabla^{t^2g_1}u|^2_{t^2g_1}-qu^2\Big)
\\
\ge &\int_{\Sg}v^2\int_{tM} \Big(\la_1(t^2g_1)u^2-qu^2\Big)
\\
= &\int_{\Sg}v^2\,\int_{tM} \Big(t^{-2}\la_1(g_1)u^2-qu^2\Big)
\\
\ge &\int_{\Sg}v^2\,\int_{tM} \Big(t^{-2}\la_1(g_1) u^2-\norm{q}_\infty u^2\Big)
\\
= &(t^{-2}\la_1(g_1) -\norm{q}_\infty)\int_{tM\times\Sg} f^2>0.
\end{split}
\end{equation}
Where in the last step we used \eqref{eq:sts0}.

\emph{Case 2:}
 $\int_{\Sg}v=0$. Then, by the previous and that $-\int_{\Sg}v(\Delta^{{\bar{g}_2}} v+qv)\ge 0$ since $\Sg$ is stable by the hypothesis and $q$ is as in \eqref{eq:sts00},  we get
\begin{equation}
\label{eq:sts3}
\begin{split}
&-\int_{tM\times\Sg}f\big(\Delta f+qf\big)
\\
=&-\int_{tM\times\Sg}v^2u\Delta^{t^2g_1} u-\int_{tM\times\Sg}u^2v\Delta^{{\bar{g}_2}} v -\int_{tM\times\Sg}qu^2v^2
\\
\ge&-\int_{tM\times\Sg}u^2v\Delta^{{\bar{g}_2}} v -\int_{tM\times\Sg}qu^2v^2
\\
= &-\int_{tM}u^2\,\int_{\Sg}v(\Delta^{{\bar{g}_2}} v+qv) \ge 0.
\end{split}
\end{equation}
Thus $tM\times \Sg$ is stable for all $t<t_0$.
\end{proof}

\begin{remark}
\label{rem:anstbl}
Note that, by \eqref{eq:sts0}, if $t>t_0$ then $tM\times \Sg$ is unstable (take $u$ to be an eigenfunction corresponding to $\lambda_1(g_1)$). Hence, such product hypersurfaces cannot be isoperimetric.
\end{remark}
\begin{remark}
\label{rem:strsbl}
The arguments of the previous proof yield that if $\Sigma$ is strictly stable, then, for all $0<t<t_0$, the product hypersurface $tM\times\Sigma$ is strictly stable.
\end{remark}

\section{Proof of the main result}

In this section, we establish the main theorem.

Following Chodosh–Engelstein–Spolaor \cite{rqi}, we denote by 
$C^{k,\alpha}_\B(\Sg_t)$ and $L^2_\B(\Sg_t)$ 
the closed subspaces of $C^{k,\alpha}(\Sg_t)$ and $L^2(\Sg_t)$
consisting of functions with zero average on $\Sg_t$.
In particular, the tangent space $T_0 \B(\Sg_t)$ of the
volume–constrained Banach manifold $\B(\Sg_t)$ can be identified
with $C^{2,\alpha}_\B(\Sg_t)$.

Let $\Sg\subset N$ be a constant mean curvature stable hypersurface. 
Define $\Sg_t := tM\times \Sg$ and
\begin{equation}
\label{eq:defW}
\La :=\Big\{ u\in C^{0,\alpha}_{\B}(\Sg_t)\,\mid u(x,\vsg)=u(\vsg), x\in M,\vsg\in \Sg \Big\}.
\end{equation} 
The arguments in Proposition~\ref{Thm:sts} yield that $\Sg_t$ is stable for sufficiently small $t>0$ and
\begin{equation}
\label{eq:KS=}
K(\Sg_t)\subset \La, \qquad t>0 \text{ sufficiently small},
\end{equation}
where $K$ stands for the kernel of the Jacobi operator.

Let $F$ be a finite perimeter set in $tM\times N$ so that $\ptl F$ is a smooth $g_-^{t}$- graph around $\Sg_t$ with corresponding function $f$. We denote
\begin{equation}
\label{eq:defp}
\pp_{g_-^{t}}(F):=\pp(f).
\end{equation}

\begin{lemma}
\label{lem:Υ}
Let $\Sg_t$ be as above.
If $t$ is sufficiently small, there exists a neighborhood $U$ of the origin in $C_\B^{2,\alpha}(\Sg_t)$ and a map 
\begin{equation}
\label{eq:Υ1}
\Upsilon: K(\Sg_t)\cap U \longrightarrow K(\Sg_t)^\perp 
\end{equation}
such that 
\begin{equation}
\label{eq:Υ2}
\Pi_{K(\Sg_t)^\perp}\big(\nabla \pp(\zeta+\Upsilon(\zeta))\big)=0.
\end{equation}
Here $\Pi$ denotes the $L^2_\B(\Sg_t)$–orthogonal projection and $^\perp$ its orthogonal complement. 
Moreover, setting
\begin{equation}
\label{eq:Υ3}
\LLL := \{\zeta + \Upsilon(\zeta)\mid \zeta  \in U \cap K(\Sg_t)\}\subset T_0 \B(\Sg_t),
\end{equation}
we have
\begin{equation}
\label{eq:Υ4}
\LLL\subset \La
\end{equation}
\end{lemma}

\begin{proof}
The proof is modeled on Appendix~A of \cite{rqi}. We include a brief outline here to ensure the condition \eqref{eq:Υ4}
Define
\begin{equation}
\label{eq:Υ5}
\NN(\zeta):=D\pp(\zeta)+\Pi_{K(\Sg_t)}(\zeta).
\end{equation}
Then $D\NN(0)$ is an isomorphism 
$C_\B^{2,\alpha}(\Sg_t)\to C_\B^{0,\alpha}(\Sg_t)$, and the inverse function 
theorem provides neighborhoods of the origin
\begin{equation}
\label{eq:Υ6}
W\subset C_\B^{0,\alpha}(\Sg_t),\qquad 
U\subset C_\B^{2,\alpha}(\Sg_t)
\end{equation}
and a homeomorphism 
\begin{equation}
\label{eq:Υ7}
\Psi:=\NN^{-1}:W\to U.
\end{equation}

We show that the construction restricts to $\La$.  
If $\zeta\in \La$, then the perimeter functional factorizes as
\begin{equation}
\label{eq:Υ8}
\pp(\zeta)=\vv(tM)\,\pp_N(\eta), \quad\text{where}\,\zeta(x,\vsg)=\eta(\vsg).
\end{equation}
Thus
\begin{equation}
\label{eq:Υ9}
D\pp(\zeta)\in\La.
\end{equation}

Owing to \eqref{eq:KS=}, we have
\begin{equation}
\label{eq:Υ10}
\NN(\zeta)\in\La
 \qquad\text{for all }\zeta\in \La,
\end{equation}

By \eqref{eq:Υ10}, the uniqueness in the inverse function theorem implies that
\begin{equation}
\label{eq:Υ11}
\Psi(W\cap\La)=U\cap\La .
\end{equation}
After shrinking $U$ if necessary, we may assume $K(\Sg_t)\cap U\subset W$. As $K(\Sg_t)\subset\La$ for $t$ small, by \eqref{eq:KS=}, we have
\begin{equation}
\label{eq:Υ12}
\Psi(U\cap K(\Sg_t))\subset\La .
\end{equation}

Finally,
\begin{equation}
\label{eq:Υ13}
\Psi(\zeta)=\zeta+\Upsilon(\zeta),\qquad 
\Upsilon(\zeta)=\Pi_{K(\Sg_t)^\perp}\Psi(\zeta),
\end{equation}
yields \eqref{eq:Υ3} and \eqref{eq:Υ2}.
Combining \eqref{eq:Υ12} and \eqref{eq:Υ13} we obtain \eqref{eq:Υ4}
\end{proof}

\begin{theorem}
\label{thm:main}
Let $M,N$ be compact Riemannian manifolds and $\be\in (0,1)$. Assume that the boundaries of the isoperimetric regions in $N$ of volume fraction $\be$ are of class $C^{2,\alpha}$.
\begin{enum}
\item Then there exists $t_0=t_0(\be)>0$ so that  for every $0<t\le t_0$
every isoperimetric region of volume $\be\vv(t M \times N)$ in $t M \times N$, is of the form  $ M \times S$, where $S$ is an isoperimetric region in $N$.

By scaling isotropically, we get equivalently.
\item
For every $s\ge s_0$
every isoperimetric region of volume $\be \vv(M \times s N)$ in $M \times s N$, is of the form  $M \times S$, where $S$ is an isoperimetric region in $s N$ and $s_0=t_0^{-1}$.
\end{enum}
\end{theorem}

\begin{proof}
According to Proposition~\ref{prp:reduce}, it suffices to give the proof when the first factor is a round sphere~$\esf$.
Let $E_i, E$ be as in Proposition~\ref{prp:pr}. 

After an isotropic homothety, \eqref{eq:ane00}, yields
\begin{equation}
\label{eq:proof000}
\vv_{g_-^{t}}(E_i)=\vv_{g_-^{t}}(E), 
\qquad 
\pp_{g_-^{t}}(E_i)\le\pp_{g_-^{t}}(E).
\end{equation}
Moreover, by Proposition \ref{prp:pr}, $E$ is cylinderoid and $\pi_2(E)$ is an isoperimetric set in $N$, of volume fraction $\be$. Consequently the hypothesis implies that $\ptl E$ is $C^{2,\alpha}$ smooth.
Shrinking $t$ if needed so that
\[
\ptl E:=\Sg_t=t\,\esf\times \Sg
\]
satisfies Lemma~\ref{lem:Υ}, i.e.
\begin{equation}
\label{eq:proof00}
\LLL\subset \La.
\end{equation}

By Proposition~\ref{prp:prpc} we have
\begin{equation}
\label{eq:prooff}
\ptl E_i \subset [\ptl E]_{\rho_i},\quad \rho_i\downarrow 0.
\end{equation}
Define $G_i$ by
\begin{equation}
\label{eq:proof0}
\pp_{g_-^{t}}(G_i)=\min\{\pp_{g_-^{t}}(F)\mid \ptl F\subset [\ptl E]_{\rho_i},\ \vv_{g_-^{t}}(F)=\vv_{g_-^{t}}(E)\}.
\end{equation}
Since 
\begin{equation}
\label{eq:prooffff}
\pp_{g_-^{t}}(G_i)\le\pp_{g_-^{t}}(E),
\end{equation}
and $E$ is a cylinderoid, then, due to \eqref{eq:proof000}, it suffices to prove that $G_i$ are cylinderoids, for $i\in\nn$, large enough.
By \cite[Theorem 4.3]{fus}, $\ptl G_i$ are $C^{2,\alpha}$ and converge to $\ptl E$ in the $C^{2,\alpha}$ topology. Let $h_i$ be the corresponding function of the graph  of $\ptl G_i$ around $\Sg_t$. Note that $h_i\in U$, for large  $i\in\nn$, where $U$ is as in Lemma \ref{lem:Υ}. 
Decompose
\begin{equation}
\label{eq:proof1}
h_i^\perp:=h_i-h_i^{\LLL}\in K(\Sg_t)^\perp.
\end{equation}
Due to Lemma \ref{lem:Υ}
\begin{equation}
\label{eq:proof2}
D\pp(h_i^{\LLL})[h_i^\perp]=0,
\end{equation}
and by stability
\begin{equation}
\label{eq:proof3}
D^2\pp(0)[h_i^\perp,h_i^\perp]\ge c \|h_i^\perp\|^2_{W^{1,2}}.
\end{equation}
By the continuity of the Hessian, we have, for sufficiently large $i\in\nn$,
\begin{equation}
\label{eq:proof4}
D^2\pp(h_i^{\LLL})[h_i^\perp,h_i^\perp]\ge \tfrac{c}{2} \|h_i^\perp\|^2_{W^{1,2}}.
\end{equation}
Taylor expansion yields
\begin{equation}
\label{eq:proof5}
\pp(h_i)-\pp(h_i^{\LLL})
= D\pp(h_i^{\LLL})[h_i^\perp] + D^2\pp(h_i^{\LLL})[h_i^\perp,h_i^\perp] + o(\|h_i^\perp\|^2_{W^{1,2}}).
\end{equation}
Combining \eqref{eq:proof3}–\eqref{eq:proof5}, for $i\in\nn$ large, we obtain
\begin{equation}
\label{eq:proof6}
\pp(h_i)-\pp(h_i^{\LLL})
\ge \tfrac{c}{4}\|h_i^\perp\|^2_{W^{1,2}}.
\end{equation}
Splitting,
\begin{equation}
\label{eq:proof7}
\pp(h_i)-\pp(0)
= \big(\pp(h_i)-\pp(h_i^{\LLL})\big) + \big(\pp(h_i^{\LLL})-\pp(0)\big).
\end{equation}
Owing to \eqref{eq:proof00}, each $h_i^{\LLL}$ corresponds to the boundary of a cylinderoid enclosing the same volume as $E$. Thus, due to (ii) of Proposition \ref{prp:pr},
\begin{equation}
\label{eq:proof8}
\pp(h_i^{\LLL})\ge\pp(0).
\end{equation}
Together with \eqref{eq:prooffff}, \eqref{eq:proof5}–\eqref{eq:proof8}, this yields
\begin{equation}
\label{eq:proof9}
h_i^\perp=0.
\end{equation}
Hence $\ptl G_i$ are cylinderoids and, consequently, the sets $G_i$ themselves are cylinderoids.
\end{proof}

\section{Appendix}
In this appendix we recall the fibrewise rearrangement used throughout the paper.
We consider here a symmetrization first introduced by Ros~\cite{ros} and later generalized by Morgan~\cite{morpol} for warped product manifolds with density.
Let $M,N$ be compact Riemannian manifolds. Let $\esf$ be a round sphere with the same volume and dimension as $M$ and with isoperimetric profile not exceeding that of $M$. Fix a point $p_0\in \esf$ and consider the foliation of the geodesic balls centered at $p_0$.
Let  $E\subset M\times N$ be a set of finite perimeter. Replace each non-empty horizontal slice with the unique geodesic ball in the foliation centred at $p_0$ having the same volume. We denote by $\sym E \subset \esf\times N$ this set, which we call the symmetrized set of~$E$. Then, there holds
\begin{equation}
\label{eq:sym1}
\vv_{\esf\times N}(\sym E)=\vv_{M\times N} (E)\quad \text{and}\ \quad \pp_{\esf\times N}(\sym E)\le \pp_{M\times N}(E).
\end{equation}
Consequently,
\begin{equation}
\label{eq:sym2}
I_{M\times N}\ge I_{\esf\times N}
\end{equation}

A Riemannian manifold $M$ endowed with constant density $a$ will be denoted by $M_a$. Then, by definition, the following equalities hold for every finite perimeter set $E\subset M$.
\begin{equation}
\label{eq:den}
\vv_{M_a}(E)=a\,\vv_{M}(E)\quad \text{and} \quad \pp_{M_a}(E)=a\,\pp_{M}(E)
\end{equation}

The following proposition, inspired by \cite[Theorem 3.12.]{morpol}, reduces the proof of the main theorem from the general case $M\times N$ to the particular case $\esf\times N$, and constitutes a fundamental ingredient of the article.

\begin{proposition}
\label{prp:reduce}
If Theorem \ref{thm:main} is true when the first factor is a round sphere, then it holds in general, i.e, when the first factor is any compact Riemannian manifold.
\end{proposition} 

\begin{proof}
We have
\begin{equation}
\label{eq:reduce1}
I_M(v)\ge \ga_2v^{(m-1)/m},\quad v\in [0,\vv(M)/2].
\end{equation}
Where $\ga_2=\ga_2 (M,\vv(M)/2)$ is the constant from \eqref{eq:Ine2} and $m\in\nn$ is the dimension of $M$.
Let $\esf$ be the $m$-dimensional round sphere so that $\vv(\esf)=\vv(M)$.

It is well known that the isoperimetric profile of a round sphere is below that of the Euclidean space of the same dimension, namely
\begin{equation}
\label{eq:reduce2}
I_{\esf}(v)\le c_mv^{(m-1)/m},
\end{equation}
where $c_m$ is the isoperimetric constant of the $m$-dimensional Euclidean space.
Let $\alpha>0$ be large enough such that
\begin{equation}
\label{eq:reduce3}
\alpha^{1/m} \ga_2>c_m
\end{equation}
and 
\begin{equation}
\label{eq:defatau}
\tau=\alpha^{-1/m}
\end{equation}
By the definition of $\tau$, there holds
\begin{equation}
\label{eq:reduce4}
\quad \vv((\tau M)_{\alpha})=\vv(M)=\vv(\esf). 
\end{equation}

Let $v\in (0,\vv(M)/2)$. According to \eqref{eq:est},  inequality \eqref{eq:reduce1} is invariant under homotheties. Thus, due to \eqref{eq:den}, \eqref{eq:reduce1} and \eqref{eq:reduce3}, we get
\begin{equation}
\label{eq:reduce5}
I_{(\tau M)_{\alpha}}(v)=\alpha I_{\tau M}\Big( \frac{v}{\alpha}\Big)\ge\alpha^{1/m} \ga_2v^{(m-1)/m}> c_m v^{(m-1)/m}\ge I_{\esf}( v).
\end{equation}
Since $v\in (0,\vv(M)/2)$ was arbitrary and the isoperimetric profile is symmetric with respect to the middle volume, 
then the previous inequality holds in $\big(0,\vv(\esf)\big)$. Namely
\begin{equation}
\label{eq:reduce6}
I_{(\tau M)_{\alpha}}>I_{\esf},\quad\text{in}\quad \big(0,\vv(M)\big).
\end{equation}
Thus, due to \eqref{eq:sym1} and \eqref{eq:reduce4}
\begin{equation}
\label{eq:reduce6a}
I_{(\tau M)_{\alpha}\times N}\ge I_{\esf\times N},\,\,\text{in}\,\, \big(0,\vv(\esf)\vv(N)\big).
\end{equation}

To complete the proof, we argue by contradiction. Assume that for some volume fraction $0<\beta<1$ in $\esf\times N$, all isoperimetric regions are cylinderoids, but there exists an isoperimetric region $E\subset \tau M\times N$, of the same volume fraction $\beta$, that is not. Since
\begin{equation}
\label{eq:reduc10}
(\tau M)_{\alpha}\times N=(\tau M\times N)_{\alpha}
\end{equation}
it follows that $E$ is an isoperimetric region in $(\tau M)_{\alpha}\times N$, say of volume $v_0$. Consequently, $\sym E$ is not a cylinderoid either.  Let $S\subset N$ be an isoperimetric region of volume $v_0\,\vv(M)^{-1}$. Then due to \eqref{eq:reduce6a} we get
\begin{equation}
\label{eq:reduc7}
\vv((\tau M)_{\alpha})\,\pp(S)\ge I_{(\tau M)_{\alpha}\times N}(v_0)\ge I_{\esf\times N}(v_0)=\pp_{\esf\times N}(\esf\times S)=\vv(\esf)\,\pp(S).
\end{equation}
Hence, by \eqref{eq:reduce4}, we obtain 
\begin{equation}
\label{eq:reduc8}
I_{(\tau M)_{\alpha}\times N}(v_0)= I_{\esf\times N}(v_0).
\end{equation}
Thus
\begin{equation}
\label{eq:reduc9}
\pp_{\esf\times N}(\sym E )\le\pp_{(\tau M)_{\alpha}\times N} (E)=I_{(\tau M)_{\alpha}\times N}(v_0)= I_{\esf\times N}(v_0).
\end{equation}
This is a contradiction, since $\sym E \subset \esf\times N$ 
would be an isoperimetric region of volume~$v_0$ that is not a cylinderoid.

\end{proof}

\bibliography{product}

\end{document}